\documentclass[12pt]{article}
\usepackage[T2A]{fontenc}
\usepackage[english]{babel}
\usepackage[utf8]{inputenc}
\usepackage{hyperref}
\usepackage{amssymb,amsmath,amsfonts,amsthm,amscd,latexsym,indentfirst,verbatim,xcolor}
\usepackage{enumerate}

\newtheorem{lem}{Lemma}[section]

\newtheorem{theorem}{Theorem}

\newtheorem{prop}[lem]{Proposition}
\newtheorem*{prob*}{Problem}
\theoremstyle{remark}
\newtheorem{rem}{Remark}

\newcommand{\Aut}{\operatorname{Aut}}
\newcommand{\Out}{\operatorname{Out}}
\newcommand{\prk}{\operatorname{prk}}
\newcommand{\Soc}{\operatorname{Soc}}

\def\ov{\overline}

\author{N.~Yang, I.B.~Gorshkov, A.M.~Staroletov, A.V.~Vasil$'$ev}
\title{On recognition of direct powers of finite simple linear groups by spectrum\thanks{The first and fourth authors were supported by Foreign Experts program in Jiangsu Province (No. JSB2018014)}\phantom{ }\thanks{The second, third, and fourth authors were supported by RAS Fundamental Research Program, project FWNF-2022-0002}}
\date{}

\begin{document}
\maketitle
\newcommand{\Addresses}{{
		\bigskip\noindent
		\footnotesize
		Nanying~Yang, \textsc{School of Science, Jiangnan University, Wuxi, 214122, P.R.~China;}\\\nopagebreak
		\textit{E-mail address: } \texttt{yangny@jiangnan.edu.cn}
		
		\medskip\noindent
		Ilya~B.~Gorshkov, \textsc{Sobolev Institute of Mathematics, Novosibirsk, Russia;}\\\nopagebreak
                \textit{E-mail address: } \texttt{ilygor8@gmail.com}
		
		\medskip\noindent
		Alexey~M.~Staroletov, \textsc{Sobolev Institute of Mathematics, Novosibirsk, Russia;}\\\nopagebreak
		\textit{E-mail address: } \texttt{staroletov@math.nsc.ru}
		
		\medskip\noindent
		Andrey~V.~Vasil$'$ev, \textsc{Sobolev Institute of Mathematics, Novosibirsk, Russia;}\\\nopagebreak
		\textit{E-mail address: } \texttt{vasand@math.nsc.ru}
		
		\medskip
}}

\begin{abstract}
The spectrum of a finite group is the set of its element orders. We give an affirmative answer to Problem 20.58(a) from the {\em Kourovka Notebook} proving that for every positive integer $k$, the $k$-th direct power of the simple linear group $L_{n}(2)$ is uniquely determined by its spectrum in the class of finite groups provided $n$ is a power of $2$ greater than or equal to $56k^2$.
\end{abstract}
\section{Introduction}

All groups considered in this paper are finite, the simple sporadic groups and simple groups of Lie type are denoted according to the notation of {\em Atlas of finite groups}~\cite{atlas}, the symmetric and alternating groups of degree $n$ are denoted by $Sym_n$ and $Alt_n$, respectively.

Given a group $G$, denote by $\omega(G)$ the spectrum of $G$, that is the set of all its element orders. Groups whose spectra coincide are said to be isospectral.
We refer to a group $G$ as {\it recognizable} (by spectrum) if every finite group isospectral to $G$ is
isomorphic to $G$, as {\it almost recognizable} (by spectrum) if there is only a finite number of pairwise nonisomorphic groups isospectral to $G$, and as {\em unrecognizable} otherwise. It is known that if a finite group is almost recognizable, then its socle is a direct product of nonabelian simple groups \cite[Lemma~1]{nonrecog}. On the other hand, if $G$ is a nonabelian simple group, then in most cases $G$ is almost recognizable \cite[Theorem~1.1]{GrechVas15}. More information on the recognition of simple groups and related topics can be found in the recent survey article~\cite{survey}.

Though the recognition problem is solved for most of the simple groups, very little is known about recognizability of (nontrivial) direct products of simple groups. The recognizability of groups $Sz(2^7)\times Sz(2^7)$ and $J_4\times J_4$ was proved in \cite{Maz97} and \cite{GorMas}, respectively. Recently, it has been proved in \cite{Suzuki} that the direct squares of Suzuki groups $Sz(q)$, where $q\geq 8$ and $q\neq32$, are recognizable and the group $Sz(32)\times Sz(32)$ is almost recognizable. For cubes of simple groups, it is only known that the group $L_{n}(2)\times L_n(2)\times L_n(2)$ is recognizable for all $n=2^l\geq64$~\cite{Gor22}.

Can a recognizable group  be a direct product of arbitrary many simple groups? If we do not presuppose that all simple factors are isomorphic, then the answer is affirmative as shown in \cite[Theorem~19]{survey}: for every $k>1$ there exists a set $\Delta(k)$ of $k$ primes such that the group $\prod\limits_{p\in\Delta(k)}Sz(2^p)$ is recognizable. If we fix a simple group $L$ (or even any finite group) and consider its direct powers, then as easily seen (cf. \cite[Section~4.3]{survey}), there is $k_0$ depending on $L$ such that $\omega(L^k)=\omega(L^{k_0})$ for all $k\geq k_0$, in particular, $L^k$ is unrecognizable for every such~$k$. The remaining question is if one can, vice versa, start with an integer $k$ and find an appropriate simple group~$L$.

\begin{prob*}{\em\cite[Problem~4.8]{survey}, \cite[Problem~20.58(a)]{Kourovka}}
Is it true that for every $k$ there is a recognizable group that is the $k$-th direct
power of a nonabelian simple group?
\end{prob*}

In the present paper we develop techniques from \cite{Gor22} and obtain an affirmative answer to this problem.

\begin{theorem}\label{t:main} Let $k$ and $l$ be positive integers and $n=2^l\geq56k^2$. Suppose that $L=L_{n}(2)$ and $P$ is the $k$-th direct power of~$L$.
If $G$ is a finite group with $\omega(G)=\omega(P)$, then $G\simeq P$.
\end{theorem}

\begin{rem}\label{r:n0} Suppose that given a positive integer $k$, one wish to find the smallest $n_0$ such that $P=L_n(2)^k$ is recognizable for every $n=2^l\geq n_0$. Theorem~\ref{t:main} provides an upper bound on $n_0$, which is quadratic. It is not hard  to show (and we do this in the last section, see Proposition~\ref{p:upper}) that there is a linear lower bound: the group $P$ is unrecognizable for all $n<2k$. The exact value of $n_0$ as a function of $k$ is not known for all $k>1$ (if $k=1$, then it follows from \cite[Corollary~1]{ZavMaz}) that $n_0=4$).
\end{rem}


The next theorem shows that the situation described in Theorem~\ref{t:main} is quite specific. Namely, for a wide range of simple groups of arbitrarily large dimension even their squares or cubes are unrecognizable by spectrum. The standard abbreviations $L_n^+(q)=L_n(q)$ and $L_n^-(q)=U_n(q)$ are used; the cyclic group of order $r$ is denoted by $\mathbb{Z}_r$.

\begin{theorem}\label{p:1}
Let $n\geq2$ be an integer and $q$ a power of a prime~$p$. The following hold.
\begin{enumerate}[{\em(i)}]
\item If $L=L^\varepsilon_n(q)$, where $\varepsilon\in\{+,-\}$,
and there exists a prime $r$ dividing $q-\varepsilon1$ and coprime to $n$, then
$\omega(L^3)=\omega(L^3\times\mathbb{Z}_r^m)$ for every positive integer $m$.
Moreover, if $n-1$ is not a power of $p$, then
$\omega(L^2)=\omega(L^2\times\mathbb{Z}_r^m)$ for every positive integer $m$.
\item If $L=S_{2n}(q)$ and $q$ is odd, then
$\omega(L^3)=\omega(L^3\times\mathbb{Z}_2^m)$ for every positive integer $m$.
Moreover, if $2n-1$ is not a power of $p$, then
$\omega(L^2)=\omega(L^2\times\mathbb{Z}_2^m)$ for every positive integer $m$.
\end{enumerate}
\end{theorem}

The paper is organized as follows.
In Section~2, we discuss arithmetic properties of spectra of simple groups alongside with some number theoretic facts.
In Section 3, we list auxiliary group theoretic results that are used in our proofs. Section 4 is devoted to the proof of Theorem~\ref{t:main}.
Finally, in Section~5, we prove Theorem~\ref{p:1} and Proposition~\ref{p:upper}.

\section{Preliminaries: arithmetic of spectra of simple groups}


Given a nonzero integer $n$, we put $\pi(n)$ for the set of prime divisors of $n$.
If $G$ is a group and $g\in G$, then we write $\pi(G)$ for $\pi(|G|)$ and $\pi(g)$ for $\pi(|g|)$.
Denote $$\rho(k)=\max\{|\pi(G)|~|~G\text{ is solvable and }|\pi(g)|\leq k\text{ for every }g\in G\}.$$

\begin{lem}\label{l:Zhang}
The following hold:
\begin{enumerate}[{\em(i)}]
 \item $\rho(k)\leq\frac{k(k+3)}{2}$ for every $k\geq1;$
 \item $\rho(k)\leq6k$ for every $k\geq1;$
 \item $\rho(k)\leq7k-9$ for every $k\geq2$.
\end{enumerate}
\end{lem}
\begin{proof}
Item (i) is \cite[Theorem~1]{Zhang}, while item (ii) follows from \cite[Theorem~1.1]{Yang}.
Item~(iii) is an easy consequence of (i) and (ii). Indeed, if $k>9$, then (ii) yields $\rho(k)\leq6k\leq7k-9$. On the other hand, $k(k+3)/2\leq7k-9$ is equivalent to $(k-2)(k-9)\leq0$, so $\rho(k)\leq7k-9$ for $2\leq k\leq9$ in view of~(i).
\end{proof}

\begin{rem} 
In the recent preprint \cite{Keller}, the authors established that $\rho(k)\leq 5k$.
\end{rem}

The prime graph (or the Gruenberg-Kegel graph) $\Gamma(G)$ of a group $G$ is defined as follows.
The vertex set is the set $\pi(G)$. Two vertices corresponding to distinct primes $r$ and $s$ are adjacent in
$\Gamma(G)$ if and only if $rs\in\omega(G)$. Recall that an independent set of vertices or a {\it coclique} of a graph $\Gamma$ is any subset of pairwise nonadjacent vertices of $\Gamma$. We write $t(G)$ to denote the greatest size of a coclique in $\Gamma(G)$. The following obvious observation is a key to our technique and will be repeatedly used in further considerations.

\begin{lem}\label{l:key}
Suppose that $\Omega$ is a coclique of size $t$ in the prime graph of a group~$G$. Then for every positive integer $k<t$, the $k$-th direct power $G^k$ of $G$ does not contain an element of order equal to the product of all primes from $\Omega$.
\end{lem}

For a set of nonzero integers $n_1,\ldots,n_k$,
we denote by $(n_1,\ldots,n_k)$ and $[n_1,\ldots,n_k]$
their greatest common divisor and least common multiple, respectively.
If $n$ is a nonzero integer and $r$ is an odd
prime with $(r, n) = 1$, then $e(r, n)$ denotes the multiplicative order of $n$ modulo $r$. Given an odd integer $n$, we
put $e(2, n) = 1$ if $n\equiv1\pmod{4}$, and $e(2,n)=2$ otherwise.
Fix an integer $a$ with $|a|>1$. A prime $r$ is said to be a primitive prime divisor of $a^i-1$ if $e(r,a)=i$. We
write $r_i(a)$ to denote some primitive prime divisor of $a^i-1$, if such a prime exists, and $R_i(a)$ to denote the
set of all such divisors. For $i\neq2$ the product of all primitive divisors of $a^i-1$ taken with multiplicities is denoted by $k_i(a)$. Put $k_2(a)=k_1(-a)$. It is well known that that primitive prime divisors exist for almost all pairs $(a,i)$.

\begin{lem}{\em\cite{Bang,Zhigmondy}}
Let $a$ be an integer and $|a|>1$. For every positive integer $i$ the set $R_i(a)$ is nonempty, except for the pairs $(a, i)\in\{(2, 1), (2, 6), (-2, 2), (-2, 3), (3, 1), (-3, 2)\}$.
\end{lem}

Sometimes it is convenient to consider primitive divisors $r_i(-q)$ instead of $r_i(q)$ (e.g., for unitary groups).
The following lemma helps to deal with the numbers $e(r_i(-q), q)$ and $e(r_i(q), -q)$.

\begin{lem}{\em\cite[Lemma~1.3]{Vas15}} Let $a$ and $i$ be integers with $|a|>1$ and $i>0$. If $i$ is odd then $k_i(-a)=k_{2i}(a)$,
and if $i$ is a multiple of $4$ then $k_i(-a)=k_i(a)$.
\end{lem}

For convenience, given a classical group $L$, we put $\prk(L)$ to denote its dimension if $L$ is a linear or unitary group, and its Lie rank if $L$ is a symplectic or orthogonal group.

Define the following function on positive integers:

\begin{center}\label{eq:eta}
$\eta(k)=\left\{\begin{array}{l}
k,\text{ if }k \text{ is odd},\\
k/2,\text{ if }k\text{ is even.}\end{array}\right.$
\end{center}

Following~\cite{Vas15}, we introduce a function $\varphi$ in order to unify further arguments.
Namely, given a simple classical group $L$ over a field of order $q$ and a prime $r$ coprime to $q$,
we put
\begin{center}\label{eq:varphi}
$\varphi(r,L)=\left\{\begin{array}{l}
e(r,\varepsilon q),\text{ if }L=L_n^{\varepsilon}(q),\text{ where }\varepsilon\in\{+,-\};\\
\eta(e(r,q)),\text{ if }L\text{ is symplectic or orthogonal.}\end{array}\right.$
\end{center}

\begin{lem}{\em\cite[Lemma 2.4]{Vas15}}\label{l:adj}
Let $L$ be a simple classical group over a field of order $q$ and characteristic $p$, and let $\prk(L)=n\geq4$.
\begin{enumerate}[{\em(i)}]
\item If $r\in\pi(L)\setminus\{p\}$, then $\varphi(r, L)\leq n$.
\item If $r$ and $s$ are distinct primes from $\pi(L)\setminus\{p\}$ with $\varphi(r, L)\leq n/2$ and
$\varphi(s, L)\leq n/2$, then $r$ and $s$ are adjacent in $\Gamma(L)$.
\item If $r$ and $s$ are distinct primes from $\pi(L)\setminus\{p\}$ with $n/2<\varphi(r,L)\leq n$
and $n/2<\varphi(s,L)\leq n$, then $r$ and $s$ are adjacent in $\Gamma(L)$ if and only if
$e(r,q)=e(s,q)$.
\item If $r$ and $s$ are distinct primes from $\pi(L)\setminus\{p\}$ and $e(r,q) = e(s,q)$,
then $r$ and $s$ are adjacent in $\Gamma(L)$.
\end{enumerate}
\end{lem}

Item (iv) of the previous lemma can be generalized as follows (cf. \cite[Lemma~2.13]{Vas15}).

\begin{lem}\label{l:2.13}
Let $m$ be a positive integer and $L$ a simple classical group over a field of order $q$ and characteristic $p$.
For $j=1,\ldots,m$, suppose that pairwise distinct primes $r_j$ lie in $\pi(L)\setminus\{p\}$ and put
$i_j=e(r_j,q)$. If $i_1,i_2,\ldots,i_m$ are greater than $2$ and pairwise distinct, then
$r_1r_2\cdots r_m\in\omega(L)$ if and only if
$k_{i_1}(q)k_{i_2}(q)\cdots k_{i_m}(q)\in\omega(L)$.
\end{lem}

We need the descriptions of the spectra for linear, unitary and symplectic simple groups obtained in~\cite{But08,But11}.

\begin{lem}{\em\cite[Corollary~3]{But08}}\label{l:spec}
Let $G=L^{\varepsilon}_{n}(q)$, where $n\geq2$, $\varepsilon\in\{+,-\}$, and $q$ is a power of a prime $p$. Put $d=(n, q-\varepsilon1)$.
Then $\omega(G)$ consists of all divisors of the following numbers:
\begin{enumerate}[{\em(i)}]
\item{$\frac{q^n-(\varepsilon1)^n}{d(q-\varepsilon1)}$;}
\item{$\frac{[q^{n_1}-(\varepsilon1)^{n_1}, q^{n_2}-(\varepsilon1)^{n_2}]}{(n/(n_1,n_2),q-\varepsilon1)}$ for $n_1, n_2 >0$ such that $n_1+n_2=n$;}
\item{$[q^{n_1}-(\varepsilon1)^{n_1}, q^{n_2}-(\varepsilon1)^{n_2},\ldots, q^{n_s}-(\varepsilon1)^{n_s}]$ for $s\geq3$ and for
$n_1, n_2,\ldots,n_s>0$ such that $n_1 + n_2 +\ldots+ n_s = n$;}
\item{ $p^{k}\cdot\frac{q^{n_1}-(\varepsilon1)^{n_1}}{d}$ for $k, n_1 > 0$ such that $p^{k-1} + 1 + {n_1}=n$;}
\item{$p^k\cdot[q^{n_1}-(\varepsilon1)^{n_1}, q^{n_2}-(\varepsilon1)^{n_2},\ldots, q^{n_s}-(\varepsilon1)^{n_s}]$ for $s\geq 2$ and $k, n_1, n_2,\ldots, n_s > 0$ such that $p^{k-1} + 1 + n_1 + n_2 +\ldots + n_s = n$;}
\item $p^k$ if $p^{k-1} + 1 = n$ for $k > 0$.
\end{enumerate}
\end{lem}

\begin{lem}{\em\cite[Corollary~2]{But11}}\label{l:spec2}
Let $G=S_{2n}(q)$, where $n\geq2$ and $q$ is a power of an odd prime number $p$.
Then $\omega(G)$ consists of all divisors of the following numbers:
\begin{enumerate}[{\em(i)}]
\item{$\frac{q^n\pm1}{2}$;}
\item{$[q^{n_1}+\varepsilon_11, q^{n_2}+\varepsilon_21,\ldots, q^{n_s}+\varepsilon_s1]$ for all $s\geq2$, $\varepsilon_i\in\{+,-\}$, $1\leq i\leq s$, and positive $\{n_j\}$,
with $n_1+n_2+\ldots+n_s= n$;}
\item{$p^k\cdot[q^{n_1}+\varepsilon_11, q^{n_2}+\varepsilon_21,\ldots, q^{n_s}+\varepsilon_s1]$ for all $s\geq1$, $\varepsilon_i\in\{+,-\}$, $1\leq i\leq s$, and positive $k$ and
$\{n_j\}$, with $p^{k-1}+1+2n_1+2n_2+\ldots+2n_s=2n$};
\item{$p^k$ if $p^{k-1}+1=2n$ for some $k>1$.}
\end{enumerate}
\end{lem}

The following two lemmas are almost direct corollaries of the above descriptions.

\begin{lem}\label{l:pow2} Suppose that $L$ is a simple classical group over a field of odd characteristic. If $\prk(L)\geq2^k+2$ for an integer $k$, then $2^{k+2}\in\omega(L)$.
\end{lem}
\begin{proof} Suppose that $u$ is the order of the underlining field of~$L$.
It follows from Lemmas~\ref{l:spec}, \ref{l:spec2} and \cite[Corollaries~6, 8, 9]{But11}
that $u^{2^k}-1\in\omega(L)$.
Note that $u^2-1$ is divisible by~8. Since $u^{2^k}-1=(u^2-1)\prod\limits_{i=1}^{k-1}(u^{2^i}+1)$,
we infer that $2^{k+2}$ divides $u^{2^k}-1$, so $2^{k+2}\in\omega(L)$.
\end{proof}

For a real number $x$, denote by $[x]$ the integral part of $x$ that is the largest integer less than or equal to~$x$.  

\begin{lem}\label{l:spec Ln(2)} Suppose that $L=L_n^\varepsilon(q)$, where $\varepsilon\in\{+,-\}$. Then the following hold:
\begin{enumerate}[{\em(i)}]
\item if $n\geq 12$, then $\Omega=\{r_i(\varepsilon{q})~|~n/2<i\leq n\}$ is a coclique of size $[(n+1)/2]$ in $\Gamma(L);$
\item if $n=2^l$ and $q$ is even, then $2^l\in\omega(L)$ and $2^{l+1}\not\in\omega(L);$
\item if $a\in\omega(L)$, then $a\leq q^n/(q-1)$.
\end{enumerate}
\end{lem}
\begin{proof}
The first assertion is a consequence of \cite[Prop.~6.9 and Table~8]{VasVd05}.
The second and the third ones follow from Lemma~\ref{l:spec}(vi) and \cite[Lemma~1.3]{GMV2009}, respectively.
\end{proof}

In three following lemmas, we concentrate on properties of the spectra of linear groups over field of order 2 and their direct products.

\begin{lem}\label{l:ga}
If $G=\Aut(L)$, where $L=L_n(2)$, $n=2^l\geq4$, then $2^{l+1}\in\omega(G)$.
\end{lem}
\begin{proof}
Observe that $L=GL_n(2)$ is the full general linear group and $G=L\rtimes\langle g\rangle$, where $g$ is the invert-transpose automorphism of $L$. Let $x$ be a unipotent matrix from $L$ whose Jordan normal form consists of $2^{l-1}-1$ Jordan blocks of size one and one block of size $2^{l-1}+1$. Clearly, $|x|=2^l$. It follows from \cite[Theorem~2.3.1]{Wall} that there exists $h\in L$ such that $x=hh^g=(hg)^2$. Then  $|hg|=2^{l+1}$, as required.
\end{proof}

\begin{lem}\label{l:M(i)}
Suppose that $i,k\geq2$ are integers and $n$ is a power of $2$. If $n\geq i$ and $n>18(k+1)$, then there exists a set $\Psi(i)$ consisting of $2k$ distinct integers $n/2<i_1,i_2,\ldots,i_k\leq n$ and $n/3<j_1,j_2,\ldots,j_k<n/2$ such that $i_m+j_m=n$ for every $1\leq m\leq k$ and if $a,b\in\Psi(i)$ with $a\neq b$, then $a$, $b$, and $i$ do not divide each other.
\end{lem}
\begin{proof}
We will find a set $\Psi(i)$ of required numbers among $2(k+1)$ numbers of the form
$$
i_m=n/2+\sum_{l=1}^mt_l\mbox{ and } j_m=n/2-\sum_{l=1}^mt_l,\mbox{ where }t_l\in\{1,2,3\}\mbox{ and }m=1,\ldots,k+1.
$$

Clearly, $i_m+j_m=n$ for all $m=1,\ldots,k+1$. Since $n>18(k+1)$, it follows that $i_m\leq n/2+3m\leq n/2+3(k+1)<n$ and $j_m\geq n/2-3m\geq n/2-3(k+1)>n/3$ for every $m=1,\ldots,k+1.$

Let us now check the desired divisibility conditions. If one of the elements of $\Psi(i)$ divides the other, say, $a$ divides $b$, then it is clear that $a=n/2-\sum_{l=1}^xt_l$, $b=n/2+\sum_{l=1}^yt_l$ for some $x,y\in\{1,\ldots,k+1\}$, and $b=2a$. Since $\sum_{l=1}^{k+1}t_l\leq3(k+1)$, we arrive to a contradiction, because $a=2b$ yields the impossible inequality $n/2\leq9(k+1)$.

There is $t_1\in\{1,2,3\}$ such that $i_1=n/2+t_1$ and $j_1=n/2-t_1$ are both not divisible by~$i$. Indeed, if not, there exists $\varepsilon\in\{+,-\}$ such that $i$ divides $n/2+\varepsilon1$ and $n/2+\varepsilon3$ simultaneously. The latter is possible only if $i=2$, but $n/2+1$ and $n/2-1$ are odd. Assume by induction on $m$ that we have already proved the same for all $i_s$ and $j_s$ with $s<m$. Then again there is $t_m\in\{1,2,3\}$ such that $i_m=i_{m-1}+t_m$ and $j_m=j_{m-1}-t_m$ are both not divisible by~$i$. Otherwise, there exists $\varepsilon\in\{+,-\}$ such that $i$ divides $i_{m-1}+\varepsilon1$ and $i_{m-1}+\varepsilon3$ simultaneously. Again this implies that $i=2$, but either $i_{m-1}\pm1$ or $i_{m-1}\pm2$ are both odd.

Now we claim that there exists at most one index $m\in\{1,\ldots, k+1\}$ such that $i_m$ or $j_m$ divides $i$. By construction, we know that $i_m>n/2>j_m>n/3$, so the only possibility is $i=2j_m$. Excluding, if necessary, the pair $i_m$ and $j_m$ with this property, we obtain $2k$ required numbers.
\end{proof}

\begin{lem}\label{l:P(i)}
Suppose that numbers $i$, $k$, $n$, and the set $\Psi(i)$ are as in Lemma~{\em\ref{l:M(i)}}.
If $\Pi(i)=\prod\limits_{j\in\Psi(i)}r_j(2)$, then the following hold:
\begin{enumerate}[{\em(i)}]
\item $\Pi(i)\in\omega(L_n(2)^k);$
\item if $\Pi(i)\in\omega(L_{n_1}(2)\times L_{n_2}(2)\times\ldots\times L_{n_k}(2))$, where $4\leq n_m\leq n$ for $m\in\{1,\ldots,k\}$, then $n_m=n$ for every~$m;$
\item $2\cdot \Pi(i)\not\in\omega(L_n(2)^k)$ and if $i\neq 6$, then $r_i(2)\cdot \Pi(i)\not\in\omega(L_n(2)^k).$
\end{enumerate}
\end{lem}
\begin{proof}
Recall that $\Psi(i)=\{i_1,\ldots, i_k, j_1,\ldots, j_k\}$, where $i_1<i_2<\ldots<i_k$ and $n=i_1+j_1=\ldots=i_k+j_k$. It follows from Lemma~\ref{l:spec}(ii) that $r_{i_m}(2)\cdot r_{j_m}(2)\in\omega(L_n(2))$ for $m=1,\ldots,k$. Hence $\Pi(i)\in\omega(L_n(2)^k)$, which proves~(i).

To prove (ii), put $L_m=L_{n_m}(2)$ for $m=1,\ldots,k$, and $P=L_1\times\ldots\times L_k$. Let $g=g_1\cdot\ldots\cdot g_k$ be an element of order $\Pi(i)$ in~$P$ and $g_m\in L_m$ for $m=1,\ldots,k$. Since $i_1,\ldots, i_k$ are greater than $n/2$ and do not divide each other, Lemma~\ref{l:adj} implies that the primes $r_{i_1}(2),\ldots, r_{i_k}(2)$ divide the orders of pairwise distinct elements among $g_1,\ldots, g_k$. Hence up to reordering, we may  assume that $r_{i_m}(2)$ divides $|g_m|$ for $m=1,\ldots,k$.

If $r_{j_1}(2)$ divides $|g_m|$ for $m>1$, then $r_{j_1}(2)r_{i_m}(2)\in\omega(L_m)$, which is impossible, because $i_m+j_1>n$ and $i_m$ and $j_1$ do not divide each other. Therefore, $r_{j_1}(2)r_{i_1}(2)$ divides $|g_1|$. In particular, we have $n_1=n$, because $i_1+j_1=n$. Furthermore, as easy to deduce from Lemma~\ref{l:spec},  $|g_1|=r_{j_1}(2)r_{i_1}(2)$. It follows now by induction on $k$ that $|g_m|=r_{j_m}(2)r_{i_m}(2)$ for each $m=2,\ldots,k$. Thus, $L_m=L_n(2)$ for all $m$, as required.

Proving (iii), we apply very similar arguments. Let $P=L_1\times\ldots\times L_k$, where $L_m=L_n(2)$ for all $m$, and let $g=g_1\cdot\ldots\cdot g_k$ be an element of order $r\cdot\Pi(i)$ in $P$, where $g_m\in L_m$ for $m=1,\ldots,k$ and either $r=2$ or $r=r_i(2)$. Arguing as in the previous paragraph, we obtain that $|g_m|=r_{j_m}(2)r_{i_m}(2)$ for each $m=1,\ldots,k$, so the order of $g$ must be exactly~$\Pi(i)$, which proves (iii).
\end{proof}

We complete the section with two simple number-theoretic observations.

\begin{lem}\label{l:composite} If  $n\geq59$ is an integer, then there are at least $[2n/3]+3$ composite numbers among $n,n+1,\ldots,2n.$
\end{lem}
\begin{proof}
Denote $I=\{n,n+1,\ldots, 2n\}$.
If $x$ and $y$ are positive integers, then $y$ divides exactly $[x/y]+1$ numbers among $0,\ldots,x$. Therefore, at least $[n/2]$ numbers in $I$ are divisible by $2$, at least $[n/3]$ are divisible by $3$, and at most $[n/6]+1$ are divisible by~$6$.
Hence $2$ or $3$ divide at least $[n/2]+[n/3]-[n/6]-1$ numbers in $I$. Note that $I$ contains 60 consecutive integers, so there exist $a,b,c,d\in I$ such that
$a\equiv5\pmod{60}$, $b\equiv 25\pmod{60}$, $c\equiv 35\pmod{60}$, and $d\equiv 55\pmod{60}$. These four numbers are composite and coprime to~$6$. Suppose that $n\equiv r\pmod 3$, where $0\leq r\leq 2$.
Then we have at least $n/2-1/2+n/3-r/3-n/6-1+4=2n/3-r/3+5/2\geq [2n/3]-1/3+5/2>[2n/3]+2$ composite numbers in~$I$. Since this number is an integer, the result follows.
\end{proof}

\begin{lem}\label{l:eta-estim} Suppose that $x\geq 0$ and $\Theta=\{ i\in\mathbb{Z},i\geq1~|~ \eta(i)\leq x \}$. Then $\frac{3x}{2}-\frac{3}{2}\leq |\Theta|\leq\frac{3x}{2}+\frac{1}{2}$.
\end{lem}
\begin{proof}
There are two cases depending on the parity of $[x]$.
Suppose that $[x]=2k$. Then $\Theta=\{1,2,\ldots,2k, 2k+2,\ldots, 4k\}$, so
$|\Theta|=3k$. Since $x-1<2k\leq x$, it follows that $(3x)/2-3/2<3k\leq (3x)/2<(3x)/2+1/2$.

If $[x]=2k+1$, then $|\Theta|=3k+2$. We have $x-1<2k+1\leq x$, so $(3x)/2-3/2<3k+3/2\leq (3x)/2$.
Then $|\Theta|>3k+3/2>(3x)/2-3/2$ and, on the other hand, $|\Theta|-1/2=3k+3/2\leq(3x)/2$, as required.
\end{proof}

\section{Preliminaries: group theory}

The first well-known lemma helps to deal with groups having trivial solvable radical.

\begin{lem}{\em\cite[3.3.20]{Robinson}}\label{l:rob}
Let $R\simeq R_1\times\cdots\times R_k$,
where $R_i$ is the $n_i$-th direct power of a simple group
$S_i$, and $S_i\not\simeq S_j$ for $i\neq j$.
Then $\Aut R\simeq\Aut{R}_1\times\cdots\times\Aut{R}_k$
and $\Aut{R}_i\simeq(\Aut S_i)\wr Sym_{n_i}$,
where in this wreath product $\Aut S_i$
appears in its right regular representation and the symmetric group $Sym_{n_i}$
in its natural permutation representation.
Moreover, these isomorphisms induce isomorphisms $\Out{R}\simeq\Out{R}_1\times\cdots\times\Out{R}_k$
and $\Out{R}_i\simeq(\Out{S}_i)\wr Sym_{n_i}$.
\end{lem}

The next two lemmas provide an existence of appropriate large solvable subgroups.

\begin{lem}\label{l:r-primes} Let $G$ be a group
and $(A_1,\ldots, A_m)$ a tuple of all chief factors of $G$.
Suppose that there exist $k$ distinct integers $i_1,i_2,\ldots,i_k$ and distinct primes $p_1,p_2,\ldots, p_k$ such that $1\leq i_j\leq m$ and $p_j$ divides $|A_{i_j}|$ for every $j\in\{1,\ldots,k\}$. Then $G$ includes a solvable subgroup $K$ such that $\pi(K)=\{p_1,\ldots,p_k\}$.
\end{lem}
\begin{proof}
We proceed by induction on $|G|$. If $G$ is simple, then $k=1$ and a Sylow $p_1$-subgroup of $G$ fits as~$K$. Thus, $G$ includes a proper nontrivial minimal normal subgroup~$M$. If $|M|$ is coprime to $\pi=\{p_1,p_2,\ldots,p_k\}$, then, by induction, $G/M$ includes a required $\pi$-subgroup and, by the Schur--Zassenhaus theorem, so does $G$.

Let now $\pi\cap\pi(M)\neq\varnothing$. Without loss of generality, we may assume that $M=A_1$. If $R$ is a Sylow $p_1$-subgroup of $M$, then $G=N_G(R)M$ due to the Frattini argument. Since $N_G(R)$ satisfies the hypothesis of the lemma, $G=N_G(R)$ by inductive arguments. In view of the Jordan--H\"{o}lder theorem, the groups $A_2,\ldots, A_k$ are among the chief factors of $\ov{G}=G/R$. Therefore, $\ov{G}$ includes a solvable subgroup $\ov{K}$ with $\pi(\ov{K})=\{p_2,\ldots,p_k\}$. The preimage $K$ of $\ov{K}$ in $G$ is a required solvable subgroup.
\end{proof}

\begin{lem}\label{l:K-solvable}
Suppose that $K$ is a nontrivial normal subgroup of a group $G$ and $r\in\pi(K)$. Then there exists a subgroup $H$ in $G$
with a normal solvable subgroup $M$ such that $H/M\simeq G/K$, $\pi(M)\subseteq\pi(K)$, and $M$ is a product of its Sylow subgroups $T_1,\ldots, T_k$, where $T_1$ is an $r$-group, and $T_j\subseteq N_M(T_i)$ for every $j>i$.
\end{lem}
\begin{proof}
We start with a Sylow $r$-subgroup $T_1$ of $K$. It follows from
the Frattini argument that $N_G(T_1)/N_K(T_1)\simeq G/K$.
By induction on $\rvert G\rvert$, we may assume that $G=N_G(T_1)$. Put $\ov{G}=G/T_1$ and $\ov{K}=K/T_1$. If $K=T_1$, then we are done, so $\ov{K}$ and $\ov{G}$ satisfies the lemma hypothesis for some prime divisor $s$ of $\ov{K}$. Therefore, the conclusion of lemma holds for some subgroups $\ov{H}$ and $\ov{M}$ of $\ov{G}$. The preimages $H$ and $M$ of these subgroups in $G$ are as required.
\end{proof}


\begin{lem}\label{l:centralizer}
Let $G$ be a group, $A\leq\Aut(G)$, and let $N$ be an $A$-invariant normal subgroup of $G$ with $(|A|,|N|)=1$.
Then $C_{G/N}(A)$ is the image of $C_G(A)$ in $G/N$.
\end{lem}

\begin{proof} Apply \cite[Corollary~3.28]{Isaacs} and the Feit--Thompson theorem.
\end{proof}


The following lemma shows that the simple linear groups $L_n(2)$ are saturated with various Frobenius subgroups.

\begin{lem}\label{l:frob}
The group $L_n(2)$, $n\geq2$, includes a Frobenius subgroups with kernel of order $2^n-1$ and cyclic complement of order $n$, as well as
a Frobenius subgroup with kernel of order $2^k$ and cyclic complement of order $2^k-1$ for every $2\leq k\leq n-1$.
\end{lem}
\begin{proof}
Both assertions are well known, see, e.g., \cite[Lemma~5]{GrVas} and \cite[Lemma~2.5]{Grech15}.
\end{proof}

The last two lemmas show how one can use Frobenius subgroups to deal with spectra of group extensions.


\begin{lem}{\em\cite[Lemma~1]{Maz97-2}}\label{l:frob-action} Let $P$ be a normal $p$-subgroup of a finite group $G$ and let $G/P$ be a Frobenius group with kernel $F$ and cyclic complement $C$. If $p$ does not divide $F$ and
$F\not\subseteq PC_G(P)/P$, then $G$ contains an element of order $p|C|$.
\end{lem}

%

\begin{lem}\label{l:action}
Suppose that $K$ is a normal subgroup of $G$ and $G/K\simeq S_1\times S_2\times\ldots\times S_m$, where $S_i$ are nonabelian simple groups.
Suppose that each $S_i$ includes a Frobenius subgroup $X_i$ which kernel is a $p_i$-group for some prime $p_i$ and complement is of prime order $s_i\not\in\pi(K)$, where $s_i$ is not a Fermat prime.
If $r\in\pi(K)$, primes $s_1,\ldots,s_m$ are pairwise distinct, and $r\not\in\{p_1,\ldots,p_m\}$, then $rs_1\cdots s_m\in\omega(G)$.
\end{lem}

\begin{proof}
By Lemma~\ref{l:K-solvable}, we may assume that $K$ is a product of its Sylow
subgroups $T_1,\ldots, T_k$ such that $T_1$ is an $r$-group and $T_j\leq N_{K}(T_i)$ for every $i$, $j $ with $1\leq i<j\leq k$. Factoring $G$ and $K$ by the Frattini subgroup $\Phi(T_1)$, we arrive at a situation where $T_1$ is elementary abelian. If $T_1\cap Z(G)\neq 1$, then there is nothing to prove, so $T_1$ acts faithfully by conjugation on~$G$ and can be considered as a subgroup of $\Aut(G)$.

We proceed by induction on~$m$ and begin with $m=1$. If $C_G(T_1)K/K\neq 1$, then $C_G(T_1)K/K=S_1$ and, clearly, there is an element in $G$ whose order equals $rs_1$.
Therefore, $C_G(T_1)\leq K$ and there exists a Hall $r'$-subgroup $J$ of $C_G(T_1)$. Then $C_G(T_1)=T_1\times J$
and $J$ is normal in $G$, because $C_G(T_1)$ is a normal in~$G$.
By Lemma~\ref{l:centralizer}, the images of $C_G(T_1)$ and $T_1$ in $G/J$ coincide.
Hence we may assume that $C_G(T_1)=T_1$.

Recall that $s_1\not\in\pi(K)$ and at most one of $T_i$ is a $p_1$-group. Therefore, applying consequently the Schur--Zassenhaus theorem to the preimages of $X_1$ in factors $G/(T_1\cdots T_i)$ for $i=k,k-1,\ldots,1$ we obtain  that there exists a subgroup $X=Y\rtimes\langle g\rangle$ of $G/T_1$, where $Y$ is a $p_1$-group, $|g|=s_1$ and $X_1$ is an image of $X$ in $G/K$. The subgroup $[Y,g]$ is the preimage of a Frobenius kernel of $X_1$, so $[Y,g]\neq1$. The action of $X$ on $T_1$ is faithful, because $C_G(T_1)=T_1$. It follows now from the cross-characteristic analogue of the Hall--Higman theorem (see, e.g., \cite[Lemma~3.6]{Vas15}) that $C_{T_1}(g)\neq1$, so $rs_1\in\omega(G)$.

Let $m\geq2$. By the above arguments, there exists an element $g\in G$ of order $s_1$ such that $C_{T_1}(g)\neq1$ and the image $g_1$ of $g$ in $G/K$ lies in~$S_1$.
Lemma~\ref{l:centralizer} implies that $C_G(g)K/K=C_{G/K}(g_1)$. Hence $C_G(g)/(K\cap C_G(g))$ being isomorphic to $C_G(g)K/K$ includes a subgroup $S$ isomorphic to $S_2\times\cdots\times S_m$.
Applying the inductive hypothesis to the preimage of $S$ in $C_G(g)$, we obtain an element of order $rs_2\cdots s_m$ in $C_G(g)$. Thus, $rs_1s_2\cdots s_m\in\omega(G)$, as required.
\end{proof}

\section{Proof of Theorem 1}
According to the hypothesis of Theorem~\ref{t:main}, $L=L_n(2)$, where $n=2^l\geq56k^2$, and $G$ is a finite group isospectral to $P=L^k$. By \cite{GrVas}, we may assume that $k\geq2$ and thereby $n\geq2^8$. We fix a chief series $G_0=1\lhd G_1\lhd\ldots\lhd G_s=G$ of $G$, where $A_i=G_i/G_{i-1}$ is a nontrivial minimal normal subgroup of $G/G_{i-1}$ for $1\leq i\leq s$.

Primitive prime divisors $r_i(2)$ for $2\leq i\leq n$ and $i\neq6$ are denoted by~$r_i$.
Put $\Omega=\{ r_{n/2+1}, r_{n/2+2},\ldots, r_n\}$ and observe that $\Omega$ is a coclique of size $n/2$ in the prime graph~$\Gamma(L)$ due to Lemma~\ref{l:spec Ln(2)}. It is also clear from the definition of $r_i$ and Fermat's little theorem that $r>n/2\geq2^7$ for every $r\in\Omega$.

\begin{lem}\label{l:groups-NK}
There exists a normal subgroup $N$ of $G$ such that $\overline{G}=G/N$
includes a normal subgroup $K$ satisfying the following conditions:
\begin{enumerate}[{\em(i)}]
\item $C_{\overline{G}}(K)=1;$
\item  $K=S_1\times S_2\times\ldots\times S_m$, where $m\leq k$ and each $S_i$ is a nonabelain simple group with $|\pi(S_i)\cap\Omega|>k;$
\item there exists $\Delta\subseteq\Omega$ such that $|\Delta|\geq 3n/8+8k$ and each $p\in\Delta$ is coprime to $|N|\cdot|\overline{G}/K|$.
\end{enumerate}
\end{lem}
\begin{proof}
Suppose that for every chief factor $A_i$ of $G$ the set $\pi(A_i)\cap\Omega$ has at most $k$ elements. By Lemma~\ref{l:Zhang}(ii), $|\Omega|=n/2\geq28k^2>k\cdot\rho(k)$.
Therefore, one can choose at least $|\Omega|/k>\rho(k)$ distinct primes in $\Omega$ dividing orders of different chief factors of $G$. Lemma~\ref{l:r-primes} implies that $G$ includes a solvable subgroup $H$ of order divisible by each of these primes. It follows from Lemma~\ref{l:key} that $|\pi(h)|\leq k$ for every $h\in H$. Since $|\pi(H)|>\rho(k)$, we arrive at a contradiction with the definition of $\rho(k)$.

Let $N$ be a normal subgroup of $G$ of the largest possible order such that $|\pi(A)\cap\Omega|\leq k$ for every chief factor $A$ of~$N$. By above, $\overline{G}=G/N\neq1$. We claim that $N$ and $K=\Soc(\overline{G})$ satisfy the conclusion of the theorem.

Firstly, $K\simeq M_1\times\ldots\times M_t$, where $M_j$ are the minimal normal subgroups of $\overline{G}$. By the choice of $N$, we have $|\pi(M_j)\cap\Omega|>k\geq2$ for each $j\in\{1,\ldots,t\}$. Since every subgroup $M_j$ is characteristically simple in $G$, it must be a direct product of groups isomorphic to some nonabelian simple group $R_j$. In particular, $K\cap C_{\overline{G}}(K)=1$.
On the other hand, $K$ includes all minimal normal subgroups of $\overline{G}$ and hence $C_{\overline{G}}(K)=1$, which proves~(i).

Since $M_j$ is a direct power of $R_j$, it follows that $\pi(R_j)=\pi(M_j)$ for each $j\in\{1,\ldots,t\}$. Therefore, $|\pi(S_i)\cap\Omega|>k$ for every composition factor $S_i$ of~$K=S_1\times\ldots\times S_m$. Lemma~\ref{l:key} yields $m\leq k$, so (ii) follows.

We claim that $|\pi(N)\cap\Omega|\leq k\cdot\rho(k)$. Assume the opposite and consider chief factors $T_1,\ldots, T_y$ of $N$ such that
$\pi(N)\cap\Omega\subseteq\pi(T_1)\cup\ldots\cup\pi(T_y)$ and $y$ is minimal.
Since $|\pi(T_i)\cap\Omega|\leq k$ for each factor $T_i$, $i=1,\ldots,y$, it follows that there are more than $\rho(k)$ factors $T_i$ having pairwise distinct primes from $\Omega$. Lemma~\ref{l:r-primes} implies that
$N$ includes a solvable subgroup $M$ with $|\pi(M)\cap\Omega|>\rho(k)$.
Therefore, there exists $g\in M$ such that $|\pi(g)\cap\Omega|>k$, which contradicts Lemma~\ref{l:key}.

Now we prove that at most $k$ primes from $\Omega$ divide $|\overline{G}/K|$. Since $C_{\overline{G}}(K)=1$, Lemma~\ref{l:rob} implies that $\overline{G}/K$ is isomorphic to a subgroup of $(\Out(S_1)\times\ldots\times\Out(S_m))\cdot Sym_m$. Since $r>n/2>k\geq m$ for every $r\in\Omega$, it follows that $$\pi(\ov{G}/K)\cap\Omega\subseteq\bigcup_{i=1}^m\pi(\Out(S_i)).$$
Thus, it suffices to prove that $|\pi(\Out(S_i))\cap\Omega|\leq1$ for all factors~$S_i$. If not, then without loss of generality we may suppose that $r$ and $r'$ are two distinct primes from $\pi(\Out(S_1))\cap\Omega$. The simple group $S_1$ must be a group of Lie type, because $2\not\in\Omega$. Let the underlying field of $S_1$ be of order $u=v^d$, where $v$ is a prime.

If $rr'$ divides $d$, then $S_1$ contains an element of order greater or equal to $$(u-1)/2\geq(2^{rr'}-1)/2>(2^{n^2/4}-1)/2>2^n,$$
which contradicts Lemma~\ref{l:spec Ln(2)}(iii). Therefore, only one of the primes $r$ and $r'$ can divide the order of the field automorphism of $S_1$. It follows that $S_1$ is a linear or unitary group and at least one of $r$ and $r'$ divides the dimension $n_1$ of $S_1$. By Lemma~\ref{l:spec}, the group $S_1$ includes an element of order at least $u^{n_1-2}-1$. If $rr'$ divides $n_1$, then  $u^{n_1-2}-1\geq(2^{rr'-2}-1)>2^n$. If $r'$ divides $n_1$ and $r$ does not, then $r$ divides $d$ and $u^{n_1-2}-1\geq(2^{r(r'-2)}-1)>2^n$. In the both cases, we again arrive at a contradiction to Lemma~\ref{l:spec Ln(2)}(iii).

Put $\Delta=\Omega\setminus(\pi(N)\cup\pi(\overline{G}/K))$. Then $|\Delta|\geq|\Omega|-k\rho(k)-k=n/2-k(\rho(k)+1)$.
Lemma~\ref{l:Zhang}(iii) yields  $7k\geq\rho(k)+9$, so $n/8\geq 7k^2\geq k(\rho(k)+9)$.
As easily seen, the latter inequality is equivalent to $n/2-k(\rho(k)+1)\geq 3n/8+8k$.
This implies that $|\Delta|\geq 3n/8+8k$, as required.
\end{proof}

Now we fix the subgroup $N$ of $G$, subgroups $K$, $S_1,\ldots,S_m$ of $\ov{G}=G/N$, and subset $\Delta$ of $\Omega$ as in Lemma~\ref{l:groups-NK}, and denote $\Delta_i=\pi(S_i)\cap\Delta$ for $i=1,\ldots,m$.

\begin{lem}\label{l:gk}
For any $k$ distinct primes $p_1,\ldots,p_k$ from $\Delta$, there is an element $g\in K$ of order $p_1\cdots p_k$. For every such element $g$ and every $i=1,\ldots,m$, $\pi(g)\cap\pi(S_i)\neq\varnothing$.
\end{lem}
\begin{proof}
The first statement of the lemma holds, because $p_1\cdots p_k\in\omega(G)=\omega(P)$ and each $p_i$ does not divide $|N|\cdot|\ov{G}/K|$ in view of Lemma~\ref{l:groups-NK}(iii).

If $\pi(g)\cap\pi(S_i)=\varnothing$ for some $i\in\{1,\ldots,m\}$, then $g$ centralizes $S_i$ and one can take an element $h\in S_i$ of order $r\in\pi(S_i)\cap\Omega\setminus\{p_1,\ldots,p_k\}$. It follows that $rp_1\cdots p_k\in\omega(K)\setminus\omega(G)$; a contradiction.
\end{proof}

\begin{lem}\label{l:pi(S)}
For every $i=1,\ldots,m$, $|\Delta_i|>3n/8+7k$.
\end{lem}

\begin{proof}
If there exist $k$ distinct primes $p_1,\ldots,p_k$ in $\Delta\setminus\Delta_i$, then one can take an element $g\in K$ of order $p_1\cdots p_k$ according to the first statement of Lemma~\ref{l:gk}. However, this contradicts the second statement of the same lemma. Thus, $|\Delta\setminus\Delta_i|<k$, so $|\Delta_i|>|\Delta|-k\geq 3n/8+7k$.
\end{proof}

Recall that $t(G)$ denotes the maximal size of a coclique in the prime graph of a group~$G$.

\begin{lem}\label{l:t(S)}
For every $i=1,\ldots,m$, either $t(S_i)> 3n/8+6k$ or $|\pi(g)\cap\pi(S_i)|\geq2$ for every $g\in K$ with $|\pi(g)|=k$ and $\pi(g)\subseteq\Delta$.
\end{lem}

\begin{proof}
Suppose that for some $i\in\{1,\ldots,m\}$ there exists $g\in K$ with $|\pi(g)|=k$, $\pi(g)\subseteq\Delta$, and $|\pi(g)\cap\pi(S_i)|\leq1$. We claim that every two distinct primes in $\Delta_i\setminus\pi(g)$ are nonadjacent in $\Gamma(S_i)$. If this is not the case, then there is an element $h$ of order $rs\in\pi(S_i)$,  where $r$ and $s$ are distinct primes in $\Delta_i$ and at most one of them divides the order of~$g$. If $g=g_1\cdot\ldots\cdot g_m$, where $g_j\in S_j$ for $1\leq j\leq m$, and $g'=g_i^{-1}g$, then $g'\in C_K(h)$, so $|\pi(g'h)\cap\Delta|>k$, which is impossible by Lemma~\ref{l:key}. Therefore, $t(S_i)\geq|\Delta_i|-k>3n/8+6k$ by Lemma~\ref{l:pi(S)}.
\end{proof}

\begin{lem}\label{l:pi(g)=pi(h)}
Suppose that $g,h\in K$ such that $\pi(g)\cap\pi(h)=\varnothing$,
$|\pi(g)|=|\pi(h)|=k$, and $\pi(g),\pi(h)\subseteq\Delta$.
Then $|\pi(g)\cap\pi(S_i)|=|\pi(h)\cap\pi(S_i)|$ for every $i=1,\ldots,m$.
\end{lem}
\begin{proof}
Write $h=h_1\cdots h_m$ and $g=g_1\cdots g_m$, where $h_i,g_i\in S_i$ for $i=1,\ldots,m$. Assume to the contrary that there exists $i\in\{1,\ldots,m\}$ such that $|\pi(h_i)|>|\pi(g_i)|$. Since $\pi(g)\cap\pi(h)=\varnothing$, it follows that $|\pi(g')|>|\pi(g)|=k$ for $g'=g_1\cdots g_{i-1}h_ig_{i+1}\cdots g_m$ and $\pi(g')\subseteq\Delta$, which contradicts Lemma~\ref{l:key}.
\end{proof}

\begin{lem}\label{l:sporadic}
For every $i=1,\ldots,m$, the factor $S_i$ is not a sporadic group.
\end{lem}
\begin{proof}
By Lemma~\ref{l:pi(S)}, $|\Delta_i|\geq1$ for each $i=1,\ldots,m$.
If $r\in\Delta_i$, then  $r>n/2\geq128$. However, it is well known, see, e.g., \cite{atlas}, that the prime divisors of the orders of the sporadic groups are less than $100$, a contradiction.
\end{proof}

\begin{lem}\label{l:alternaing}
For every $i=1,\ldots,m$, the factor $S_i$ is not an alternating group.
\end{lem}
\begin{proof}
Suppose that one of $S_i$ is an alternating group of degree~$d$, for definiteness, $S_1\simeq Alt_d$.
To arrive at a contradiction it suffices to show that $d\geq 2n+2$, because in this case $2^{l+1}=2n\in\omega(S_1)\subseteq\omega(G)$, which is impossible by Lemma~\ref{l:spec Ln(2)}(ii).

First, we suppose that $t(S_1)>3n/8$. Then, see, e.g., \cite[Proposition~1.1]{VasVd11},
$$3n/8+1\leq t(S_1)\leq1+|\{x\leq p\leq 2x~|~p\text{ is prime}\}|,$$
where $x=[(d+1)/2]$. Since $n\geq256$, it follows that $x>59$ and Lemma~\ref{l:composite} yields
$|\{x\leq p\leq 2x~|~p\text{ is prime}\}|\leq x+1-([(2x)/3]+3)\leq x/3-4/3.$
Hence $3n/8\leq x/3-4/3$.  Then $x\geq n+4$ and, consequently, $d\geq 2n+2$, a contradiction.

Suppose now that $t(S_1)\leq 3n/8$. There exists an element $g\in K$ such that $|\pi(g)|=k$, $\pi(g)\subseteq\Delta$ and, by Lemma~\ref{l:t(S)}, $|\pi(g)\cap\pi(S_1)|\geq 2$.
Denote $\Omega'=\{r_i\in\Omega~|~i+1\text{ is composite}\}$. By little Fermat's theorem $i$ divides $r_i-1$ and, if $i+1$ is composite, then $i<r_i-1$, so $r_i-1\geq2i$. It follows that $r>2n+1$ for every $r\in\Omega'$.

By Lemma~\ref{l:composite}, the set $\{n/2+2,\ldots,2\cdot(n/2+2)\}$
contains at least $[n/3+4/3]+3$ composite numbers, so $|\Omega'|\geq n/3$.
Therefore, $$|\Omega'\cap\Delta_1|=|\Omega'|+|\Delta_1|-|\Omega'\cup\Delta_1|\geq|\Omega'|+|\Delta_1|-|\Omega|\geq n/3+3n/8+7k+1-n/2\geq2k.$$
It follows that there are at least $k$ primes in $(\Omega'\cap\Delta_1)\setminus\pi(g)$. Take an element $h=h_1\ldots h_m$ with $h_i\in S_i$, $i=1,\ldots,m$, such that $\pi(h)$ consists of these $k$ primes. Lemma~\ref{l:pi(g)=pi(h)} yields $|\pi(h_1)|\geq2$, so there are at least two primes greater than $n+1$ and which are adjacent in $\Gamma(S_1)$. It is possible only if $d\geq 2n+2$, which leads to a final contradiction.
\end{proof}

\begin{lem}\label{l:exceptional}
For every $i=1,\ldots,m$, the factor $S_i$ is not an exceptional group of Lie type.
\end{lem}
\begin{proof}
In this lemma, we use well-known information on the orders of simple exceptional groups of Lie type and their maximal tori, see, e.g., \cite[Table~6]{atlas}, \cite[Lemma 1.3]{VasVd05}, and \cite[Lemma~2.6]{VasVd11}.
Assume that $S_1$ is an exceptional group of Lie type over a field of order $u$ and characteristic~$v$.
If $S_1\in\{E_8(u), E_7(u), E_6(u), {}^2E_6(u), F_4(u)\}$, then each prime $r\in\pi(S_1)$
either is equal to $v$ or belongs to the set $R_j(u)$ of primitive prime divisors of $u^j-1$ for some integer $j$, where
in each of these cases, the number of possible indices $j$ is at most $18$. For each $R_j(u)$, there exists a maximal torus $T_j$ of $S_i$ such that $R_j(u)\subseteq\pi(T_j)$. If $S_1\not\in\{E_8(u), E_7(u), E_6(u), {}^2E_6(u), F_4(u)\}$, then $S_1$ includes (up to conjugation) at most 13 maximal tori. It is clear that the order of each of these maximal tori is divisible by at most $k$ primes from $\Delta$.  Therefore, $|\Delta_1|\leq 18k+1$ in all cases. Lemma~\ref{l:pi(S)} implies that $3n/8+7k<18k+1$, which is impossible since $n\geq 56k^2$.
\end{proof}

\begin{lem}\label{l:m=k} The equality $m=k$ holds.
\end{lem}
\begin{proof}
Lemma~\ref{l:groups-NK}(ii) yields $m\leq k$.
Assume that $m<k$. Take an element $g\in K$ such that
$|\pi(g)|=k$ and $\pi(g)\subseteq\Delta$. Then without loss of generality, we may assume that $|\pi(g)\cap\pi(S_1)|\geq 2$ and put $n_1=\prk S_1$.

By Lemmas~\ref{l:sporadic}-\ref{l:exceptional}, each group $S_i$,
$i=1,\ldots,m$, is a classical group of Lie type. Let $S_1$ be a classical group over a field of order $u$ and characteristic~$v$.
First we consider the case when $v$ is odd. If $n_1\leq3$, then
$S_1$ includes at most 13 maximal tori (up to conjugation) (see, e.g., \cite[Lemma~1.2]{VasVd05}), and we arrive at a contradiction as in Lemma~\ref{l:exceptional}. Hence $n_1\geq4$.

If $S_1\simeq L^-_{n_1}(u)$ or $S_1\simeq O^-_{n_1}(u)$, then put $\varepsilon=-$, otherwise put $\varepsilon=+$.
Each $r\in\Delta_1$ either is equal to $v$ or belongs to $R_{j}(\varepsilon u)$ for some integer $j\in\{1,\ldots,n_1\}$. Put $\Theta=\{ e(r,\varepsilon u)~|~ r\in\Delta_1, r\neq v\}$. By \cite[Lemma~2.13]{Vas15} (see also Remark~2 after it), $S_1$ has an element of order $k_j(\varepsilon{u})$ for each $j\in\Theta$. Therefore, $|R_j(\varepsilon{u})\cap\Omega|\leq k$. If $\rvert\Delta_1\rvert\geq3k+1$, then one can find $k$ primes $p_1,p_2,\ldots, p_k$ in $\Delta_1\setminus\{v\}$ with $e(p_j,\varepsilon u)>2$. If, additionally, $\rvert\Delta_1\rvert-\rvert\Theta\rvert\geq3k+1$, then for every $p_j$, $j=1,\ldots,k$, there is $p_j'\in\Delta_1\setminus\{p_1,\ldots,p_k\}$ with $e(p_j,\varepsilon u)=e(p_j',\varepsilon u)$. Take $h\in K$ such that $|h|=p_1p_2\cdots p_k$ and put $h=h_1h_2$, where $\pi(h_1)=\pi(h)\cap\pi(S_1)$ and $\pi(h_2)\cap\pi(S_1)=\varnothing$. By Lemma~\ref{l:gk}, we may assume that $h_1\neq1$ and $\pi(h_1)=\{p_1,\ldots,p_s\}$, where $1\leq s\leq k$. By the choice of $p_j$, $e(p_j, \varepsilon{u})>2$ for every $1\leq j\leq k$, so Lemma~\ref{l:2.13} implies that there is $h_1'\in S_1$ of order $p_1'p_1\cdots p_s$. Since $h_2$ centralizes $S_1$, we obtain the element $h_1'h_2\in K$ of order $p_1'p_1\ldots p_k$; a contradiction. Therefore, $|\Delta_1|-|\Theta|\leq 3k$ and hence $|\Theta|\geq |\Delta_1|-3k>3n/8+4k$ due to Lemma~\ref{l:pi(S)}.

Suppose that there are at least $2k$ elements in $\Theta$ greater than $n_1/2$.
Then there exists $h\in K$ such that $|h|=p_1\cdots p_k$, where $p_1,\ldots,p_k$ do not not belong to $\pi(g)$, $e(p_i,\varepsilon u)>n_1/2$ for $i=1,\ldots,k$, and $e(p_i,\varepsilon u)\neq e(p_j,\varepsilon u)$ for $i\neq j$.
Lemma~\ref{l:pi(g)=pi(h)} implies that $|\pi(h)\cap\pi(S_1)|\geq2$, which contradicts Lemma~\ref{l:adj}(iii). If $S_1$ is a linear or unitary group, then $\Theta\subseteq\{1,\ldots, n_1\}$ and hence $|\Theta|\leq n_1/2+2k-1$. Assume that $S_1$ is a symplectic or orthogonal group. As above, we see that the set $\{j~\in\Theta~|~n_1/2<\eta(j)\}|$ has at most $2k-1$ elements. By Lemma~\ref{l:eta-estim}, the set $\{j\in\Theta~|~\eta(j)\leq n_1/2\}$ has at least $(3n_1)/4-3/2$ elements and the set $\{j\in\Theta~|~\eta(j)\leq n_1\}$ has at most $(3n_1)/2+1/2$ elements. Hence $|\Theta|\leq 2k-1+(3n_1)/2+1/2-(3n_1)/4+3/2=2k+(3n_1)/4+1$.

Therefore, $(3n_1)/4+2k+1\geq|\Theta|\geq3n/8+4k+1$. It follows that $n_1\geq n/2+8k/3>n/2+2$.
Lemma~\ref{l:pow2} yields $2^{l+1}\in\omega(S_1)$; a contradiction with Lemma~\ref{l:spec Ln(2)}(ii).

Thus, $S_1$ is a group of even characteristic, i.e., $u=2^f$ for some integer~$f$.
Then $R_{n_1f}(2)\subseteq \pi(S_1)$ or $R_{2n_1f}(2)\subseteq \pi(S_1)$. It follows that $n_1f\leq n$.
Consider the set $\Omega'=\{r_i\in\Omega~|~i\text{ is odd}\}$ of size $n/4$. Since $|\Omega'\cap\Delta|\geq n/4+|\Delta|-n/2=|\Delta|-n/4$ and $|\Delta|\geq 3n/8+8k+1$, we have $|\Omega'\cap\Delta|>2k$. Hence, there exist $k$ distinct primes from $(\Omega'\cap\Delta)\setminus\pi(g)$ and an element $h$ in $K$ whose order equal to their product. It follows from Lemma~\ref{l:pi(g)=pi(h)} that $\pi(h)\cap\pi(S_1)\geq2$.

Therefore, there are two distinct primes $p_1,p_2\in \Omega'\cap\Delta\cap\pi(S_1)$ such that $p_1p_2\in\omega(S_1)$. By the choice of these primes, $j_1=e(p_1,2)\neq j_2=e(p_2,2)$. Put also $e_1=\varphi(p_1,S_1)$ and $e_2=\varphi(p_2,S_1)$. Since $p_1,p_2\in\Omega'$, the numbers $j_1$ and $j_2$ are odd. Note that $p_1$ divides $2^{2e_1f}-1$. Hence $j_1$ divides $2e_1f$, so $j_1$ divides $e_1f$. If $e_1f\geq 2j_1>n$, then $n_1f>n$; a contradiction. It follows that $e_1f=j_1$, so $e_1f=j_1>n/2\geq n_1f/2$ and hence $e_1>n_1/2$. Similarly, $j_2=e_2f$ and $e_2>n_1/2$. Since $p_1p_2\in\omega(S_1)$, Lemma~\ref{l:adj} yields
$e_1=e_2$. Then $j_1=j_2$, which contradicts the choice of $p_1$ and $p_2$, thus completing the proof.
\end{proof}



\begin{lem}\label{l:Ai} For every $i=1,\ldots,m$, the factor $S_i$ is a linear group over a field of order~$2$ and $7n/8+7k\leq\prk(S_i)\leq n$.
\end{lem}
\begin{proof}
By Lemma~\ref{l:gk}, there is an element $g\in K$ such that $|\pi(g)|=k$, $\pi(g)\subseteq\Delta$. For every $i=1,\ldots,m$ and any such element $g$, the same lemma yields $|\pi(g)\cap\pi(S_i)|\neq\varnothing$. Since $m=k$ in view of Lemma~\ref{l:m=k}, it follows that $|\pi(g)\cap\pi(S_i)|=1$ and, by Lemma~\ref{l:t(S)}, the inequality $t(S_i)\geq 3n/8+6k$ holds.

Consider one of the groups $S_i$, say $S_1$. Denote by $u$ the order of the underlying field of $S_1$ and put $n_1=\prk S_1$. It follows from \cite[Tables~2, 3]{VasVd11} that $t(S_1)\leq(3n_1+5)/4$. Therefore, $(3n_1+5)/4\geq 3n/8+6k$ and hence $n_1>n/2+2$.

If $u$ is odd, then we immediately arrive at a contradiction in view of Lemma~\ref{l:pow2}(ii). Assume now that $u=2^f$. If $S_1$ is a symplectic or orthogonal group, then $R_{2t}(u)\subseteq\pi(S_1)$ for every $t<n_1$. Since $n_1>n/2+2$, it follows that $R_{n+2}(u)\subseteq\pi(S_1)$, so $R_{f(n+2)}(2)\subseteq\pi(S_1)$, a contradiction. If $S_1$ is an unitary group, then $\pi(S_1)$ includes either $R_{2(n_1-1)}(u)$ or $R_{2n_1}(u)$. Hence $R_{2jf}(2)\subseteq\pi(S_1)$ for some $j\in\{n_1-1, n_1\}$. Since $2j>n$, we arrive at a contradiction. If $S_1$ is a linear group, then $R_{n_1f}(2)\subseteq R_{n_1}(u)\subseteq\pi(S_1)$. Since $2n_1>n$, it follows that $u=2$, as required.

Clearly, $\prk(S_1)\leq n$. Since $|\Omega|=n/2$ and $|\Delta_1|>3n/8+7k$ in view of Lemma~\ref{l:pi(S)}, there are at most $n/8-7k-1$ integers $a$ between $n/2+1$ and $n$ such that $e(s,2)\neq a$ for some $s\in\pi(S_1)$. Therefore, $\prk(S_i)\geq 7n/8+7k$, and we are done.
\end{proof}

\begin{lem}\label{l:n/3}
If $r\in\pi(N)\cup\pi(\ov{G}/K)$, then $e(r,2)<n/3$. In particular, $\Delta=\Omega$.
\end{lem}
\begin{proof}

Suppose that $r\in\pi(\ov{G}/K)$. Since $C_{\overline{G}}(K)=1$, we infer that $\overline{G}$ embeds into $\Aut(K)$. In view of Lemma~\ref{l:rob}, this implies that $r\in\pi(\Out(L_a(2)))$ or $r\in\pi(Sym_b)$, where $b\leq k$. Since $|\Out(L_a(2))|=2$, it follows that $e(r,2)<r\leq b\leq k<n/3$.

Suppose now that $r\in\pi(N)$ and $e=e(r,2)>n/3$. We arrive at a contradiction by showing that there is an element in $G$ whose order is a product of $k+1$ primes pairwise nonadjacent in $\Gamma(L)$. Clearly, $e$ divides at most one integer between $2n/3$ and $n$. Observe also that every $s\in\Omega\setminus\{r_n\}$ is not Fermat's prime, because $e(s,2)$ is not a power of~$2$. For $i=1,\ldots,m$, set $n_i=\prk(S_i)$ and $\Delta'_i=\Delta_i\setminus\{r\}$.

Define $\Theta=\{x\in\mathbb{N} \mid 2n/3<x<7n/8+7k\}$. Direct calculation shows that for every $i=1,\ldots,m$, there are at least $n/12+14k$ primes $s$ in $\Delta'_i$ such that $e(s,2)\in\Theta$. It follows that one can choose $m$ distinct primes $s_1,\ldots,s_m$ such that $s_i\in\Delta'_i$ and $e(s_i,2)\in\Theta$ is not a multiple of $e$ for every $i=1,\ldots,m$. Lemma~\ref{l:Ai} yields $e(s_i,2)<7n/9+7k\leq n_i$ for every $i=1,\ldots,m$, so each $S_i$ includes a Frobenius subgroup whose kernel is a 2-group and complement has order $s_i$ by Lemma~\ref{l:frob}. It follows from Lemma~\ref{l:action} that there is an element $x\in G$ of order $rs_1\cdots s_m$. All the $s_i$ in this product are nonadjacent in $\Gamma(L)$ with each other because they are different primes from $\Omega$, and with $r$ because $e+e(s_i,2)>n$ and $e$ does not divide $e(s_i,2)$. It remains to note that $|\pi(x)|=k+1$ since $m=k$ in view of Lemma~\ref{l:m=k}.
\end{proof}

\begin{lem}\label{l:SiEqL} For every $i=1,\ldots,m$, the factor $S_i$ is isomorphic to~$L$.
\end{lem}
\begin{proof}
Since $n\geq56k^2>18(k+1)$, we may consider a set $\Psi(2)$ as in Lemma~\ref{l:M(i)}. Lemma~\ref{l:P(i)}(i) yields $$\Pi(2)=\prod\limits_{j\in\Psi(i)}r_j(2)\in\omega(P)=\omega(G).$$ Moreover, $\Pi(2)\in\omega(K)$ by Lemma~\ref{l:n/3}. In view of
Lemma~\ref{l:Ai}, we have $S_i\simeq L_{n_i}(2)$, where $4\leq n_i\leq n$ for every $i=1,\ldots,m$. It follows from Lemma~\ref{l:P(i)}(ii) that $n_i=n$ for each $i$, as required.
\end{proof}

\begin{lem}
$G/N=K$.
\end{lem}
\begin{proof}
Suppose that a prime $r$ divides $|\overline{G}/K|$.
Take an element $g\in\overline{G}/K$ or order $r$.
Suppose that $S_i^g=S_i$ for every $i\in\{1,\ldots,k\}$. Since $C_{\overline{G}}(K)=1$,
we may assume that $[S_1,g]\neq 1$. Then $g$ acts on $S_1$ as a graph automorphism, so $r=2$ and $S_1\langle g\rangle\simeq\Aut(L)$. It follows from Lemma~\ref{l:ga} that $2^{l+1}\in\omega(\Aut(L))\subseteq\omega(G)$, which contradicts Lemma~\ref{l:spec Ln(2)}(ii).

Thus, $S_i^g\neq S_i$ for some $i$. Then one can assume, up to reordering, that $g$ permutes factors in the product $S_1\times S_2\times\ldots\times S_r$. Denote by $x$ an element of $L$ whose order is $r^k$, where $k$ is maximal possible. If $h=(x,1,\ldots,1)\in S_1\times S_2\times\ldots\times S_r$, then $|hg|=r^{k+1}$. Since $r^{k+1}\not\in\omega(P)$, we get a contradiction.
\end{proof}

\begin{lem}
$N=1$.
\end{lem}
\begin{proof}
Suppose that $N\neq1$. By Lemma~\ref{l:K-solvable}, we may assume that $N$ is solvable. Hence, arguing by induction on $|G|$, we may also assume that $N$ is an elementary abelian $p$-group for some $p\in\pi(G)$.

Suppose that $C_G(N)=G$. If $p=r_i$, then $i\geq2$ and we take $x=\Pi(i)$, where $\Pi(i)$ is as in Lemma~\ref{l:P(i)}. If $p=2$, then take $x=\Pi(2)$. It follows from Lemma~\ref{l:P(i)}(iii)
that $px\in\omega(G)\setminus\omega(L^k)$; a contradiction. Therefore, without loss of generality, we may assume that $S_1\cap C_G(N)/N=1$.

Assume that $p$ is odd. Let $t$ be an integer such that
$p^t\in\omega(L)$ and $p^{t+1}\not\in\omega(L)$. Lemma~\ref{l:spec} implies that
$p^t$ divides $2^j-1$ for some $j\in\{1,\ldots,n\}$. Then $i=e(p,2)$ divides $j$. Suppose that $j=n$.
Lemma~\ref{l:n/3} yields $i<n/3$. Since $n/i$ is a power of $2$ and coprime to $p$, it follows that $2^i-1$ is divisible by $p^t$ (see, e.g. \cite[Lemma~1.7]{Vas15}).
So we can assume that $1\leq j\leq n-1$. Lemma~\ref{l:frob} implies that
$S_1$ includes a Frobenius subgroup with kernel of order $2^j$ and cyclic complement of order $2^j-1$. It follows from Lemma~\ref{l:frob-action} that $p^{t+1}\in\omega(G)$; a contradiction.

If $p=2$, then Lemma~\ref{l:frob} implies that $S_1$ includes a Frobenius group with kernel of order $2^n-1$ and complement of order $n=2^l$. It follows from Lemma~\ref{l:frob-action} that $2^{l+1}\in\omega(G)$, which contradicts Lemma~\ref{l:spec Ln(2)}.
\end{proof}

\section{Proof of Theorem 2}

In this short section we prove Theorem~\ref{p:1} and Proposition~\ref{p:upper}.
If $\Delta$ is a nonempty set of integers, then
$\mu(\Delta)$ stands for the set of
all maximal elements of $\Delta$ with respect to divisibility.
Given a finite group $G$, put $\mu(G)=\mu(\omega(G))$.
Note that $\omega(G)$ consists of all the divisors of $\mu(G)$ and so is completely determined by it.
It is easy to see that if $G$ and $H$ are finite groups,
then

\begin{equation}\label{eq:mu}
\mu(G\times H)=\mu(\{[a,b]~|~a\in\mu(G),b\in\mu(H)\}).
\end{equation}

Suppose that $G$ is a finite group and $r\in\pi(G)$. If $a$ is the unique element of $\mu(G)$ coprime to $r$, then it follows from \eqref{eq:mu} that $r$ divides all the elements of $\omega(G\times G)$, so $\omega(G\times G)=\omega(G\times G\times\mathbb{Z}_r^m)$ for every positive integer $m$. Similarly, if there are only two distinct integers $a$ and $b$ from $\mu(G)$ both coprime to $r$, then $[a,b]$ is the unique element of $\mu(G\times G)$ coprime to $r$. Therefore, $r$ divides all the elements of $\mu(G^3)$ and, consequently, $\omega(G^3)=\omega(G^3\times\mathbb{Z}_r^m)$ for every~$m$. Thus, Theorem~\ref{p:1} follows from the next two lemmas.

\begin{lem} Let $L=L^\varepsilon_n(q)$, $\varepsilon\in\{+,-\}$, and
$q$ a power of a prime $p$. Suppose that there exists $r\in\pi(q-\varepsilon1)\setminus\pi(n)$. Then there are at most two elements of $\mu(L)$ coprime to $r$. Moreover, if $n-1$ is not a power of $p$, then there is only one such element.
\end{lem}
\begin{proof}
Since $r$ does not divide $d=(n,q-\varepsilon 1)$, the only integers from Lemma~\ref{l:spec} coprime to $r$ are
$\frac{q^n-(\varepsilon1)^n}{d(q-\varepsilon1)}$ and $p^{k}$. Furthermore, $p^k\in\mu(L)$ only if $n=p^{k-1}+1$, and we are done.
\end{proof}

\begin{lem}
Suppose that $L=S_{2n}(q)$, where $q$ is a power of an odd prime $p$. Then there are at most two elements of $\mu(L)$ coprime to $r$. Moreover, if $2n-1$ is not a power of $p$, then there is only one such element.
\end{lem}

\begin{proof}
We argue as in the previous lemma applying Lemma~\ref{l:spec2} instead of Lemma~\ref{l:spec}.
\end{proof}

If $n=2^l\geq2$, then $n$ is even, so $k\geq[(n+3)/2]=n/2+1$ yields $n<2k$. Therefore, the following proposition establishes the lower bound on $n_0$ from Remark~\ref{r:n0} after Theorem~\ref{t:main} in Introduction.

\begin{prop}\label{p:upper} If $L=L_n(2)$ and $k_0=[\frac{n+3}{2}]$, then $\omega(L^k)=\omega(L^{k_0})$ for every $k\geq k_0$.
\end{prop}

\begin{proof} Since $L_n(2)=1$ for $n=1$, we may assume that $n\geq2$. Denote by $2^e$ the $2$-exponent of~$L$, put $t=[(n+2)/2]$ and
denote by $N$ the least common multiple of $k_0=[(n+3)/2]$ integers $2^e, 2^t-1, 2^{t+1}-1,\ldots, 2^n-1$. Applying Lemma~\ref{l:spec}, we conclude that $N$ is the exponent of $L$. It follows from \eqref{eq:mu} by induction on $k$ that $\mu(L^k)=\mu\{[a_1,\ldots,a_k] \mid a_i\in \mu(L)\}$. Thus, $\mu(L^{k_0})=\{N\}$ is a singleton, so the proposition follows.
\end{proof}

As observed in Introduction, the exact value of $n_0$ is known only for $k=1$. For $k$ equal to $2$ and $3$, we have $n_0(2)\geq8$ and $n_0(3)\geq16$, because $\omega(L_{4}(2)^2)=\omega(L_{4}(2)\times L_{3}(2))$  and $\omega(L_{8}(2)^3)=\omega(L_{8}(2)^2\times L_{7}(2))$, respectively. It would be also interesting to know the asymptotic behavior of $n_0$ as $k$ tends to infinity.

\Addresses
\end{document}